\documentclass[11pt,a4paper]{article}
\usepackage{mymacros}

\newenvironment{ack}{\section*{Acknowledgements}}{}

\usepackage[titles]{tocloft}
\usepackage{stmaryrd}
\usepackage{cite}
\usepackage{enumerate}
\title{Polynomial rate of relaxation for the Glauber dynamics of infinite-volume critical Ising model}

\author{Haoran Hu\footnote{Courant Institute of Mathematical Sciences, New York University. E-mail: {\tt hh3169@nyu.edu}.}}

\date{\vspace*{-2em}} 

\begin{document}
\maketitle
\begin{abstract}
We consider the relaxation time for the Glauber dynamics of infinite-volume critical ferromagnetic Ising model on $\Z^{d}$ in any dimension $d\geq2$. Under the assumptions regarding the finite-volume log-Sobolev constant and the 1-arm exponent of the critical 1-spin expectation, we show that the equal-position temporal spin correlation function decays polynomially fast in time.
\end{abstract}



\section{Introduction}

\subsection{Introduction and main results}

The speed of relaxation to equilibrium is one of the most interesting problems in the study of Markov processes. In particular, we consider the Glauber dynamics for Ising models on $\Z^{d}$ at $d\geq2$, which is a family of continuous-time reversible Markov processes that model thermodynamic evolution. These processes provide a way to sample the Ising spins out of stationary distribution. For finite-volume systems, the convergence speed of such dynamics depends on the size of the spectral gap and the log-Sobolev constant of the generator, both of which are strongly related to the phase transition phenomena. Indeed, for the Ising model on $\Lambda_{L}=\llbracket-L,L\rrbracket^{d}$ at inverse temperature $\beta>0$, it is believed that in all dimensions $d\geq2$, the inverse spectral gap is of order-1 for $\beta<\beta_{c}$, polynomial in $L$ for $\beta=\beta_{c}$, and exponential in $L$ for $\beta>\beta_{c}$, see, for example, \cite{HH77,LF93,M99,WH}.

More precisely, the uniform boundedness of the spectral gap at high temperature $\beta<\beta_{c}$ is provided by weak and strong spatial mixing properties. Using a spin-spin correlation inequality, Ding, Song, and Sun \cite{DSS22} showed that Ising model in the entire regime $\beta<\beta_{c}$ satisfies strong spatial mixing on all finite subgraphs of $\mathbb{Z}^{d},d\geq2$. This result matches a previously proved cutoff phenomena by Lubetzky and Sly \cite{LS13,LS14,LS16}. For the inverse spectral gap at critical temperature, Lubetzky and Sly \cite{LS12} first established a polynomial estimate on $\Lambda_{L}$ at $d=2$. In high dimensions $d\geq5$, Bauerschmidt and Dagallier \cite{BD22} established a polynomial bound \eqref{eq:poly} on the log-Sobolev constant $\gamma_{\Lambda_{L}}$ for near-critical Ising model on $\Lambda_{L}$. The proof is based on the renormalization group approach and the Polchinski equation criterion, see \cite{BBS19,BBD23} for overviews and related discussions. 

At high temperatures, the spectrum of the finite-volume Glauber generator is bounded away from zero uniformly in the side length. As a result, the corresponding infinite volume dynamics is exponentially ergodic (see \cite[Theorem 3.3]{M99}), and hence converges exponentially fast in time. Whereas in the critical phase, motivated by the polynomial growth of the finite volume log-Sobolev constant \eqref{eq:poly}, we naturally expect an algebraic relaxation to equilibrium for the Glauber dynamics on $\Z^{d}$. However, the mixing time is expected to be polynomial at criticality, and the spins tend to have long-range interactions. Thus one needs additional information to control the correlations. In our case, we will require an upper bound \eqref{eq:arm} for the 1-spin expectation $\langle\sigma_{0}\rangle^{+}_{\Lambda_{L}}$ at the center of the cube $\Lambda_{L}$ with $+$ boundary condition.

In this paper, we will always consider the critical Glauber dynamics (defined in Section~\ref{sec:notations}) at any fixed dimension $d\geq2$. For simplicity of notations, we sometimes write $X\lesssim Y$ as shorthand for the inequality $X\leq cY$ for some constant $c>0$. Our main result is based on the following general assumptions:
\begin{enumerate}[(A)]
\item\label{polynomial} Let $\gamma_{\Lambda_{L}}$ denote the log-Sobolev constant on $\Lambda_{L}$ with arbitrary boundary condition (see \eqref{eq:LSI} for definition). For some finite exponent $\eta>0$, $\gamma_{\Lambda_{L}}^{-1}$ is at most polynomial in the side-length:
\begin{equation}\label{eq:poly}
    \gamma_{\Lambda_{L}}^{-1}\lesssim L^{\eta}.
\end{equation}  
\item\label{arm} Let $\langle\sigma_{0}\rangle^{+}_{\Lambda_{L}}$ denote the one-spin expectation on $\Lambda_{L}$ with $+$-boundary condition. For some $\delta\in(0,1]$,
\begin{equation}\label{eq:arm}
\langle\sigma_{0}\rangle^{+}_{\Lambda_{L}}\lesssim L^{-\delta}.
\end{equation}
\end{enumerate}

\begin{theorem}\label{th:main}
    Fix any $d\geq2$, let $P_{t}$ denote the infinite-volume Markov semigroup for the Ising Glauber dynamics at critical temperature $\beta_{c}$. Under Assumptions \eqref{polynomial} and \eqref{arm}, there exists a constant $C=C(\beta_{c},\eta,\delta,d,c_{M})$ such that the polynomial bound
    \begin{equation}\label{eq:result}
        \langle\sigma_{0},P_{t}\sigma_{0}\rangle_{\mathbb{Z}^{d}}\leq\frac{C}{t^{\alpha}},\quad \forall t\geq0
    \end{equation}
    holds with the exponent $\alpha$ given by 
    \begin{equation}\label{eq:alpha}
        \alpha=\frac{2\delta\wedge1}{2\eta}.
    \end{equation}
    In the above, $c_{M}>0$ is an upper bound for the flipping rates of the dynamics, see \eqref{eq:bound-rates}. $\langle\cdot\rangle_{\Z^{d}}$ denotes the unique infinite volume Ising expectation \eqref{eq:infinite-expectation} on $\Z^{d}$ at critical temperature $\beta_{c}$. $\langle f,g\rangle_{\Z^{d}}:=\langle fg\rangle_{\Z^{d}}$ is the $L^{2}$ inner product for any mean zero functions $f$ and $g$.
\end{theorem}

Here we comment on the ideas of our assumptions and main result. First, for Assumption \eqref{polynomial}, Bauerschmidt and Dagallier's proof for $d\geq5$ relies on the known mean-field bound \cite{A82,AG83} for the susceptibility on $\Z^{d}$, and the finite-volume mean-field bound for the critical susceptibility \cite{CJN21}, as detailed in \cite[Corollary~1.2]{BD22}. Their result is slightly stronger than our assumption with arbitrary boundary condition since it is valid under any general coupling constants and external fields. At dimension $d=2$, the aforementioned polynomial bound for the inverse spectral gap was proved by Lubetzky and Sly \cite{LS12}, but the bound \eqref{eq:poly} for the log-Sobolev constant required for the proof of Theorem~\ref{th:main} is still unknown. Using conformal symmetry and the method of block spin dynamics, Lubetzky and Sly proved their result by establishing a spatial mixing property of an entire face of the boundary. This compensates for the loss of the strong mixing properties of single spins at critical temperature. For dimensions $d=3,4$, no rigorous result is known at this point.

Second, for Assumption \eqref{arm}, $\delta\leq1$ was shown by Handa, Heydenreich, and Sakai \cite{HHS19} for $d\geq5$, conditional on the algebraic decay of the infinite volume two-point function:
\begin{equation}\label{eq:two-point}
    c(d)\vert x-y\vert^{2-d}\leq\langle\sigma_{x}\sigma_{y}\rangle_{\Z^{d}}\leq C(d)\vert x-y\vert^{2-d},
\end{equation}
where $\vert x-y\vert$ denotes the Euclidean distance of $x,y\in\mathbb{Z}^{d}$. The upper bound in \eqref{eq:two-point} is due to reflection positivity and the infrared bound of the nearest neighbor Ising model, see \cite{FILS78,FFS76,MMS77}. Using lace expansion techniques, the lower bound in \eqref{eq:two-point} was first established by Sakai \cite[Theorem~1.3]{Sak07} at sufficiently large dimensions $d\gg4$. In a recent work by Duminil-Copin and Panis \cite{DCP24}, they proved the lower bound for any dimension $d\geq5$ by random current expansion, together with Griffiths' inequality and MMS inequality (see \cite{MMS77}). Based on the solvability for $d=2$ \cite{WMTB} and several numerical predictions, $\delta\leq1$ is also expected to hold for $d=2,3$ and 4.

However, $\delta>0$ remains largely unsolved in any dimension. Through Edward-Sokal coupling, the finite-volume 1-spin magnetization is equivalent to the probability that 0 connects $\partial\Lambda_{L}$ in the FK-Ising percolation. In other cases, conditional on the algebraic decay of the two-point function $\P_{p_{c}}(0\leftrightarrow x)\asymp\vert x\vert^{2-d}$, Kozma and Nachmias \cite{KN11} determined the 1-arm exponent for critical Bernoulli percolation on $\Z^{d}$ in high dimensions. Using a similar method, Hulshof \cite{H15} determined the 1-arm exponent for critical long-range percolation and branching random walk. For $d>10$ dimensional critical nearest-neighbor percolations, due to the validity of lace expansions techniques, the result of Kozma-Nachmias holds, see \cite{FvdH}. Similarly, given the recent result of Duminil-Copin and Panis \cite{DCP24}, $\delta=1$ is expected to be the optimal 1-arm exponent for the $d\geq5$ critical Ising model \cite{HHS19}. It is not yet possible to study lower dimensional FK-Ising percolations, since the non-mean-field nature renders the estimate for the probability of the appearance of large clusters difficult.

The polynomial decay similar to \eqref{eq:result} was first studied by Janvresse, Landim, Quastel, and Yau in \cite{JLQY99} for the Kawasaki dynamics of zero-range processes. Following this work, Landim and Yau \cite{LY03} proved the polynomial rate of convergence with certain logarithmic correction for the Kawasaki dynamics of the Ginzburg-Landau model. Using different methods, similar results were established in \cite{CCR05,CM00} for the Kawasaki dynamics of particle systems with finite-range interactions. All of these works resulted in $O(t^{-d/2})$ relaxation speeds, as expected for the Kawasaki dynamics at high temperature. Moreover, they all managed to show convergence for any square-integrable local functions. In comparison, for the Glauber dynamics of critical Ising model, we investigate the relaxation rate by establishing a polynomial decay for the equal-position temporal spin correlation function \eqref{eq:result}. If we consider any general square-integrable local function $f$, similar polynomial convergence is also expected to hold with exponent $\alpha\propto\frac{1}{\eta}$. This generalization is difficult within our approach, due to the missing techniques for general boundary conditioning, and the estimate for the correlations between $f$ and the boundary conditions throughout the bulk of $\Lambda_{L}$. The following sections present the heuristics and proof of our main result.

\subsection{Notations and outlines of the proof}
\label{sec:notations}

Let $\Lambda$ be a finite subgraph of $\mathbb{Z}^{d}, d\geq2$, two vertices $x,y\in\Z^{d}$ are called adjacent iff $\sum_{i=1}^{d}\vert x_{i}-y_{i}\vert=1$ and this is denoted by $x\sim y$. At critical inverse temperature $\beta_{c}$, $\omega$-boundary condition, and zero external field, we define the finite volume Ising model on spin configurations $\sigma\in\Omega_{\Lambda}=\{-1,+1\}^{\Lambda}$ to be the probability measure 
\begin{equation}\label{eq:free-ising}
    \mu^{\omega}_{\Lambda,\beta_{c}}(\sigma)=\frac{1}{Z^{\omega}_{\Lambda,\beta_{c}}}\exp\Big(\beta_{c}\sum_{\substack{x\sim y\\x,y\in\Lambda}}\sigma_{x}\sigma_{y}+\beta_{c}\sum_{\substack{a\sim b\\a\in\Lambda,\;b\in\partial\Lambda}}\sigma_{a}\omega_{b}\Big).
\end{equation}
Here $\partial\Lambda=\{x\in\Z^{d};x\sim\Lambda,x\notin\Lambda\}$, $\omega\in\Omega_{\Z^{d}}=\{-1,+1\}^{\Z^{d}}$ or $\omega\equiv0$, $Z^{\omega}_{\Lambda,\beta_{c}}$ denotes the partition function. The Ising model is said to be under $+$ boundary, $-$ boundary, and free boundary condition if we take $\omega\equiv1$, $\omega\equiv-1$ and $\omega\equiv0$ in \eqref{eq:free-ising} respectively. For the free boundary case, we will omit the superscripts $\omega$. Since we will only consider critical temperature throughout the whole article, the beta subscript will also be omitted in the following context. The expectation under this finite volume measure is mostly denoted by $\E_{\Lambda}^{\omega}(\cdot)$ or $\langle\cdot\rangle^{\omega}_{\Lambda}$, and the variance is denoted by $\var_{\Lambda}^{\omega}(\cdot)$. Due to the continuity of phase transition (see \cite{ADCS15} for $d\geq3$, and Kaufman-Onsager's approach for $d=2$), for any sequence of finite regular subgraphs $\Lambda_{n}\uparrow\Z^{d}$ and boundary conditions $(\omega_{n})_{n\geq1}$, the distributions $\langle\cdot\rangle_{\Lambda_{n}}^{\omega_{n}}$ weakly converge to a unique infinite-volume Gibbs state $\langle\cdot\rangle_{\Z^{d}}$. More precisely, for any local function $F:\Omega_{\Z^{d}}\mapsto\mathbb{R}$, we have
\begin{equation}\label{eq:infinite-expectation}
    \lim_{n\to\infty}\langle F\rangle^{\omega_{n}}_{\Lambda_{n}}=\langle F\rangle_{\Z^{d}}.
\end{equation}
The Gibbs state $\langle\cdot\rangle_{\Z^{d}}$ implicitly characterizes the unique infinite-volume Ising measure on the space $\Omega_{\Z^{d}}$. For more details, see \cite{FV18}.

The Glauber dynamics of critical Ising model is a family of Markov processes determined by infinitesimal generators on certain function spaces. First, for the finite-volume case, we consider the space of square-integrable functions $L^{2}(\mu_{\Lambda}^{\omega})$. Define the self-adjoint generator $\mathcal{L}_{\Lambda}^{\omega}$ of the finite-volume Glauber dynamics by
\begin{align}\label{eq:finite-generator}
    (\mathcal{L}_{\Lambda}^{\omega}F)(\sigma):=\sum_{x\in\Lambda}c(x,\sigma)(\nabla_{x}F)(\sigma)
    =\sum_{x\in\Lambda}c(x,\sigma)[F(\sigma^{x})-F(\sigma)],\;\;\forall F\in L^{2}(\mu_{\Lambda}^{\omega}).
\end{align}
where $\sigma^{x}$ denotes the spin configuration obtained by flipping $\sigma$ at vertex $x\in\Lambda$. The coefficients $c(x,\sigma)$ are called the flipping rates, and are additionally required to satisfy the following:
\begin{enumerate}[(i)]
    \item Finite range interactions: For some fixed $r>0$ and any $x\in\Lambda$, if the spin configurations $\sigma,\sigma'\in\Omega_{\Lambda}$ agree on $\{y\in\Lambda;\max_{1\leq i\leq d}\vert y_{i}-x_{i}\vert\leq r\}$, then $c(x,\sigma)=c(x,\sigma')$.
    \item Detailed balance: For all $\sigma\in\Omega_{\Lambda}$ and $x\in\Lambda$,
    \begin{equation}
        \frac{c(x,\sigma)}{c(x,\sigma^{x})}=\exp\big(2\beta_{c}\sigma_{x}\sum_{y\sim x}\sigma_{y}\big).
    \end{equation}
    \item Positivity and boundedness: There exist positive constants $c_{m}$ and $c_{M}$ such that
    \begin{equation}\label{eq:bound-rates}
        0<c_{m}\leq\inf_{x,\sigma}c(x,\sigma)\leq\sup_{x,\sigma}c(x,\sigma)\leq c_{M}<\infty.
    \end{equation}
    \item Translation invariance: If there exist some $k\in\Lambda$ such that $\sigma_{y}=\sigma'_{y+k}$ for all $y\in\Lambda$, then $c(x,\sigma)=c(x+k,\sigma')$.
\end{enumerate}
The two most common examples for the choice of flipping rates are 
\begin{enumerate}[(i)]
    \item Metropolis: $c(x,\sigma)=\exp\big(2\beta_{c}\sigma_{x}\sum_{y\sim x}\sigma_{y}\big)\wedge1$.
    \item Heat-bath: $c(x,\sigma)=\Big[1+\exp\big(-2\beta_{c}\sigma_{x}\sum_{y\sim x}\sigma_{y}\big)\Big]^{-1}$.
\end{enumerate}
One could naturally extend this finite-volume Glauber dynamics to the infinite-volume case. Let $B(\Omega_{\mathbb{Z}^{d}})$ be a function space consisting of the family of functions $F:\Omega_{\Z^{d}}\mapsto\mathbb{R}$ such that $\interleave F\interleave<\infty$, where the triple norm is given by
\begin{equation}\label{eq:triple-norm}
    \interleave F\interleave:=\sum_{x\in\mathbb{Z}^{d}}\sup_{\sigma\in\Omega_{\mathbb{Z}^{d}}}\vert(\nabla_{x}f)(\sigma)\vert<\infty.
\end{equation}
Similar as before, we naturally define the infinite-volume Glauber generator by
\begin{equation}\label{eq:infinite-generator}
    (\mathcal{L}F)(\sigma)=\sum_{x\in\mathbb{Z}^{d}}c(x,\sigma)(\nabla_{x}F)(\sigma),\;\;\forall F\in B(\Omega_{\mathbb{Z}^{d}}).
\end{equation}
In this case, $\cL$ is an essentially self-adjoint on $L^{2}(\mu_{\Z^{d}})$ and the dynamics is reversible with respect the infinite-volume Gibbs state $\langle\cdot\rangle_{\Z^{d}}$.

For the finite-volume dynamics, the Glauber Dirichlet form associated with the generator \eqref{eq:finite-generator} is 
\begin{equation} \label{eq:Dirichlet}
  D_{\Lambda}^{\omega} (F) =-\E^{\omega}_{\Lambda}[F\cdot\mathcal{L}^{\omega}_{\Lambda}F].
\end{equation}
The infinite-volume Dirichlet form is similarly defined by replacing the generator by $\cL$. Two important tools to quantify the rate of relaxation to equilibrium are the spectral gap inequality \eqref{eq:SGI} and the log-Sobolev inequality \eqref{eq:LSI}. Following \cite{GZ03,M99,LSC97}, they are given by
\begin{equation}\label{eq:SGI}
    \var_{\Lambda}^{\omega}(F)\leq\operatorname{gap}(\mathcal{L}^{\omega}_{\Lambda})^{-1}D^{\omega}_{\Lambda}(F),\quad\forall F\in L^{2}(\mu_{\Lambda}^{\omega});\tag{SGI}
\end{equation}
\begin{equation}\label{eq:LSI}
    \ent_{\Lambda}^{\omega}(F) \leq \frac{2}{\gamma_{\Lambda}} D^{\omega}_{\Lambda}(\sqrt{F}),\quad\forall F\in L^{2}(\mu_{\Lambda}^{\omega}),\;F\geq0.\tag{LSI}
\end{equation}
Here the quantity $\ent_{\Lambda}^{\omega}(F) = \E_{\Lambda}^{\omega}\Phi(F)-\Phi(\E^{\omega}_{\Lambda} F)$ with $\Phi(x)=x\log x$
is the relative entropy. The largest non-negative constants $\operatorname{gap}(\cL^{\omega}_{\Lambda})$ and $\gamma_{\Lambda}$ such that the inequalities hold are called the spectral gap and the log-Sobolev constant respectively. In particular, we omitted $\omega$ for the notation of log-Sobolev constant since \eqref{eq:LSI} holds uniformly in the boundary conditions.

Let $P_{t}$ denote the Markov semigroup associated with the infinite-volume generator \eqref{eq:infinite-generator}, due to the reversibility of the infinite-volume dynamics, and the translation invariance of $\mu_{\Z^{d}}$, the temporal correlation is in fact a second order moment:
\begin{equation}
\label{eq:second-moment}
    \langle\sigma_{0},P_{t}\sigma_{0}\rangle_{\mathbb{Z}^{d}}=\langle(P_{t/2}\sigma_{0})^{2}\rangle_{\Z^{d}}.
\end{equation}
The averaging with respect to the infinite-volume measure $\mu_{\Z^{d}}$ in \eqref{eq:second-moment} can be approximated using any large finite-volume subgraphs. We consider a $d$-dimensional discrete torus $\T_{L}$ with side-length $2L$. It consists of all the vertices in $\Z^{d}$ modulo the relation that $x,y$ are identical iff $\sum_{i=1}^{d}\vert (x_{i}-y_{i})\mod 2L\vert=0$. Roughly speaking, $\T_{L}$ is obtained from gluing the opposite edges of the discrete cube $\Lambda_{L}$. The critical Ising measure $\mu_{\T_{L}}=\langle\cdot\rangle_{\T_{L}}$ without external field is said to be under periodic boundary condition. The reason why torus $\T_{L}$ is of interest in our context is that it maintains the translation invariance property, which holds for $\langle\cdot\rangle_{\Z^{d}}$ but not for $\langle\cdot\rangle_{\Lambda_{L}}$.

\begin{proposition}\label{th:finite-poly}
    Fix any $d\geq2$ and $L\gg1$, let $P^{\T_{L}}_{t}$ denote the semigroup of the critical Glauber dynamics on the $d$-dimensional torus $\T_{L}$. Under the same setting as Theorem \ref{th:main}, we have the polynomial rate of decay
    \begin{equation}\label{eq:finite}
        \langle(P^{\T_{L}}_{t}\sigma_{0})^{2}\rangle_{\T_{L}}\leq\frac{C_{0}}{t^{\alpha}},\qquad\forall t\geq0.
    \end{equation}
   Here $C_{0}=C_{0}(\beta_{c},\eta,\delta,d,c_{M})$, and $\alpha$ is given by the formula \eqref{eq:alpha}.
\end{proposition}

\begin{proof}[Proof of Theorem \ref{th:main}]
    Since the infinite-volume measure and the Glauber semigroup can both be approximated by the finite-volume ones, using an argument similar to the proof of \cite[Theorem 2.1]{LY03}, we have
    \begin{equation}
        \langle (P_{t}\sigma_{0})^{2}\rangle_{\Z^{d}}=\lim_{L\to\infty}\langle (P^{\T_{L}}_{t}\sigma_{0})^{2}\rangle_{\T_{L}}.
    \end{equation}
    The desired result \eqref{eq:result} is easily drawn because the constant $C_{0}$ in the upper bound \eqref{eq:finite} is independent of the length $L$.
\end{proof}

For the proof of Proposition \ref{th:finite-poly}, we follow a strategy similar to the one developed in \cite{JLQY99, LY03}. Heuristically speaking, by domain Markov property, we try to decouple the large dynamics on $\T_{L}$ into a product dynamics on several smaller blocks $\cup_{i=1}^{q}\Lambda^{i}_{\ell}\approx\T_{L}$. Here $q\sim(L/\ell)^{d}$ is the number of the small blocks, and the side length $\ell(t)$ is chosen to grow with an appropriate polynomial rate in time such that it balances with the LSI constant $\gamma_{\Lambda_{\ell}}$. However, our situation is different from the conservative dynamics considered by Janvresse, Landim, Quastel, and Yau, where fixing the number of particles in each small block produces independence for different blocks and ergodicity of the canonical measure. Instead, due to the nearest-neighbor interaction of the Ising measure, we need to fix the spin configurations on the boundary subgraph $\cup_{i=1}^{q}\partial\Lambda^{i}_{\ell}$ to obtain mutual conditional independence. Finally, the contribution from the boundary spins is controlled by an estimate for the entropy of the conditioned distribution against the full Ising measure. 

For additional information on the critical spin configurations inside each individual small block, we need an additional bound for the averaged 1-spin expectation.

\begin{lemma}\label{th:averaged}
    Consider the critical temperature Ising model on $\Lambda_{L}=\llbracket-L,L\rrbracket^{d}$ with $+$ boundary condition in any fixed dimension $d\geq2$. Assume the assumption \eqref{eq:arm} is satisfied, then there exists a universal constant $c_{0}(\delta,d)>0$ such that
    \begin{equation}\label{eq:ave-arm}
        \frac{1}{\vert\Lambda_{L}\vert}\sum_{x\in\Lambda_{L}}\big(\langle\sigma_{x}\rangle_{\Lambda_{L}}^{+}\big)^{2}\leq c_{0}L^{-(2\delta\wedge1)}.
    \end{equation} 
\end{lemma}
\begin{proof}
    For any $0\leq i\leq L$, we define $\lambda_{i}$ as the union of faces of the cube centered at 0 with side length $2i+1$. Namely, let $\lambda_{i}:=\{x\in\Lambda_{L};\;\max\{\vert x_{1}\vert,...,\vert x_{d}\vert\}=i\}$. By monotonicity in the volume of the magnetization (see e.g., \cite[Lemma~3.22]{FV18}), we have
    \begin{equation}\label{eq:averaged-fkg}
        \frac{1}{\vert\Lambda_{L}\vert}\sum_{x\in\Lambda_{L}}\big(\langle\sigma_{x}\rangle_{\Lambda_{L}}^{+}\big)^{2}\lesssim\frac{1}{L^{d}}\sum_{i=0}^{L}\sum_{x\in\lambda_{i}}\big(\langle\sigma_{x}\rangle_{\Lambda_{L}}^{+}\big)^{2}\leq\frac{1}{L^{d}}\sum_{i=0}^{L}\sum_{x\in\lambda_{i}}\big(\langle\sigma_{x}\rangle_{x+\Lambda_{L-i}}^{+}\big)^{2}.
    \end{equation}
    Here, $x+\Lambda_{L-i}$ is obtained by translating the cube $\Lambda_{L-i}$ by $x$. According to \eqref{eq:arm}, the total contribution from $i=0$ and $i=L$ is at most $O(L^{-d-2\delta}+L^{-1})$. Thus, we only need to bound the average \eqref{eq:averaged-fkg} with a smaller summation from $i=1$ to $i=L-1$. Using \eqref{eq:arm} again and approximation by integral, we can write
    \begin{align}\label{eq:inte}
        \frac{1}{L^{d}}\sum_{i=1}^{L-1}\sum_{x\in\lambda_{i}}&\big(\langle\sigma_{x}\rangle_{x+
        \Lambda_{L-i}}^{+}\big)^{2}\lesssim\frac{1}{L^{d}}\sum_{i=1}^{L-1}\frac{i^{d-1}}{(L-i)^{2\delta}}\notag\\
        &\leq\frac{1}{L^{2\delta}}\Big(\int_{0}^{1-1/L}\frac{u^{d-1}}{(1-u)^{2\delta}}du+\frac{1}{L}\int_{0}^{1-1/L}\Big[\frac{(d-1)u^{d-2}}{(1-u)^{2\delta}}+\frac{2\delta u^{d-1}}{(1-u)^{2\delta+1}}\Big]du\Big)\notag\\
        &\lesssim\frac{1}{L^{2\delta}}\Big(\int_{1/L}^{1}\frac{dr}{r^{2\delta}}+\frac{1}{L}\int_{1/L}^{1}\frac{dr}{r^{2\delta}}+\frac{1}{L}\int_{1/L}^{1}\frac{dr}{r^{2\delta+1}}\Big).
    \end{align}
    For $\delta\in(0,1]$, the highest order of \eqref{eq:inte} is $O(L^{-(2\delta\wedge1)})$ and the lemma is concluded.
\end{proof}

\section{Proof of Proposition \ref{th:finite-poly}}

\begin{proof}[Proof of Proposition \ref{th:finite-poly}]
    Suppose $\alpha=(2\delta\wedge1)/2\eta$ as defined by \eqref{eq:alpha}. For the temporal regime $[1,\infty)$, we fix a partition by inductively taking $t_{0}=1$ and $t_{n}>2t_{n-1}$, $n\geq1$. Given a closed interval $[1,T]$ with $T\gg1$ arbitrarily large but fixed, let $n(T)$ be the largest integer such that $t_{n(T)-1}<T$, and redefine $t_{n(T)}:=T$. By backward Kolmogorov's equation, we have
    \begin{align}\label{eq:int}
        T^{\alpha+1}&\langle(P^{\T_{L}}_{T}\sigma_{0})^{2}\rangle_{\T_{L}}-\langle (P^{\T_{L}}_{1}\sigma_{0})^{2}\rangle_{\T_{L}}\notag\\
        &=\int_{1}^{T}\partial_{t}\Big(t^{\alpha+1}\langle(P^{\T_{L}}_{t}\sigma_{0})^{2}\rangle_{\T_{L}}\Big)dt\notag\\
        &=\sum_{i=0}^{n(T)-1}\int_{t_{i}}^{t_{i+1}}(\alpha+1)t^{\alpha}\langle(P^{\T_{L}}_{t}\sigma_{0})^{2}\rangle_{\T_{L}}dt+2\int_{1}^{T}t^{\alpha+1}\langle (P^{\T_{L}}_{t}\sigma_{0})(\cL_{\T_{L}}P^{\T_{L}}_{t}\sigma_{0})\rangle_{\T_{L}}dt\notag\\
        &=\sum_{i=0}^{n(T)-1}\int_{t_{i}}^{t_{i+1}}(\alpha+1)t^{\alpha}\langle(P^{\T_{L}}_{t}\sigma_{0})^{2}\rangle_{\T_{L}}dt-2\int_{1}^{T}t^{\alpha+1}D_{\T_{L}}(P^{\T_{L}}_{t}\sigma_{0})dt.
    \end{align}
    For any $x\in\T_{L}$, let $\tau_{x}$ denote the translation operator on $\Omega_{\T_{L}}$, i.e., $(\tau_{x}\sigma)_{y}=\sigma_{x+y}$. Since the Ising measure $\mu_{\T_{L}}$ is translation invariant, the first term in the right hand side of \eqref{eq:int} is equal to
    \begin{equation}\label{eq:1term}
        \frac{1}{\vert\T_{L}\vert}\sum_{x\in\T_{L}}\sum_{i=0}^{n(T)-1}\int_{t_{i}}^{t_{i+1}}(\alpha+1)t^{\alpha}\langle(P^{\T_{L}}_{t}\sigma_{x})^{2}\rangle_{\T_{L}}dt.
    \end{equation}

Now for any $t\in[1,T]$, define
\begin{equation}\label{eq:ell}
    \ell(t)=\left(\frac{1}{c_{1}(\alpha+1)}t_{\max\{i;\;t_{i}\leq t\}}\right)^{\frac{1}{2\eta}},
\end{equation}
where $c_{1}>0$ is the constant factor in the inequality \eqref{eq:poly}. As we can see, the function $\ell(t)$ is constant in each subinterval $[t_{i},t_{i+1})$, and $\ell^{-(2\delta\wedge1)}t^{\alpha}\lesssim1$. From then on, let $\ell$ be the shorthand for $\ell(t)$, since the time $t$ for the dynamics is implicitly understood. Using $\lfloor\ell\rfloor$ as side-length, we define a grid (evolving through time) on the torus:
\begin{equation}
    \partial\Lambda_{\ell}:=\big\{x\in\T_{L};\;\exists1\leq j\leq d,k\in\mathbb{N}\cup\{0\},\;\mathrm{s.t.,}\;\vert x_{j}\vert=(k+\frac12)(\lfloor\ell\rfloor-1)\big\}.
\end{equation}
Suppose $\Lambda_{\ell}:=\{\Lambda_{\ell}^{1},...,\Lambda_{\ell}^{q}\}$ is a family of cubes with side length $\lfloor\ell\rfloor-2$, such that the centers $\{z^{j}\}_{j=1}^{q}$ of all these cubes satisfy $z^{j}\in (\lfloor\ell\rfloor-1)\Z^{d}\cap \T_{L}$. Then one easily finds that $\partial\Lambda_{\ell}$ is the union of the boundaries of these cubes and $\T_{L}=\partial\Lambda_{\ell}\cup\Lambda_{\ell}^{1}\cup\cdots\cup\Lambda_{\ell}^{q}$. For simplicity of notation, we would like to omit the $\ell$ subscripts in the following proof.

Let $B_{\partial\Lambda}$ denote the conditional expectation operator:
\begin{equation}
    B_{\partial\Lambda}(u)(\omega)=\E_{\T_{L}}(u\vert\;\omega\vert_{\partial\Lambda}),\quad\forall u\in L^{2}(\mu_{\T_{L}}),\;\omega\in\Omega_{\T_{L}}.
\end{equation}
A simple application of the inequality $x^{2}\leq2y^{2}+2(x-y)^{2}$ yields an upper bound for \eqref{eq:1term}:
\begin{align}\label{eq:up1}
    \frac{2}{\vert\T_{L}\vert}\sum_{x\in\T_{L}}&\sum_{i=0}^{n(T)-1}\int_{t_{i}}^{t_{i+1}}(\alpha+1)t^{\alpha}\langle(B_{\partial\Lambda}P^{\T_{L}}_{t}\sigma_{x})^{2}\rangle_{\T_{L}}dt\notag\\
    &+\frac{2}{\vert\T_{L}\vert}\sum_{x\in\T_{L}}\int_{1}^{T}(\alpha+1)t^{\alpha}\langle\big[(I-B_{\partial\Lambda})P^{\T_{L}}_{t}\sigma_{x}\big]^{2}\rangle_{\T_{L}}dt,
\end{align}
where $I$ denotes the identity operator. For a fixed $\omega\in\Omega_{\T_{L}}$, the conditional expectation $B_{\partial\Lambda}(\cdot)(\omega)$ is equivalent to the expectation under the product probability measure $\mu_{\Lambda}(\cdot\vert\omega)$ defined by
\begin{equation*}
    \mu_{\Lambda}(\cdot\vert\omega)=\mu_{\Lambda^{1}}^{\omega}\otimes\cdots\otimes\mu_{\Lambda^{q}}^{\omega}.
\end{equation*}
Indeed, this is provided by the domain Markov property of spin systems, and the fact that Ising measure has nearest-neighbor interactions. After fixing the spin configuration on the boundary $\partial\Lambda$, the spins in different cubes $\Lambda^{i}\neq\Lambda^{j}$ are independent. Meanwhile, we observe that for any fixed $x\in\T_{L}$, the variable inside the angle bracket in the second term of \eqref{eq:up1} is in fact the variance of $P^{\T_{L}}_{t}\sigma_{x}$ under the measure $\mu_{\Lambda}(\cdot\vert\omega)$. According to the aforementioned reasons and the Efron-Stein inequality, we have
\begin{equation}\label{eq:var1}
    \langle[(I-B_{\partial\Lambda})P^{\T_{L}}_{t}\sigma_{x}]^{2}\rangle_{\T_{L}}\leq\sum_{j=1}^{q}\langle\var^{\omega}_{\Lambda^{j}}(P^{\T_{L}}_{t}\sigma_{x})\rangle_{\T_{L}}.
\end{equation}
Next, we apply \eqref{eq:SGI} and \eqref{polynomial} to all cubes $\Lambda^{j}$, the sum of variances is bounded by
\begin{align}\label{eq:var2}
    \operatorname{gap}(\cL^{\omega}_{\Lambda^{1}})^{-1}\sum_{j=1}^{q}\sum_{y\in\Lambda^{j}}\langle(\nabla_{y}P^{\T_{L}}_{t}\sigma_{x})^{2}\rangle_{\T_{L}}&\leq \operatorname{gap}(\cL^{\omega}_{\Lambda^{1}})^{-1}\sum_{y\in\T_{L}}\langle(\nabla_{y}P^{\T_{L}}_{t}\sigma_{0})^{2}\rangle_{\T_{L}}\notag\\
    &\leq c_{1}\lfloor\ell\rfloor^{\eta}D_{\T_{L}}(P^{\T_{L}}_{t}\sigma_{0}).
\end{align}
For the first inequality, we applied the translation invariance of the dynamics again and dropped the dependence on the starting position $x\in\T_{L}$. The second inequality follows from assumption \eqref{eq:poly}, \eqref{eq:SGI}, and the fact that LSI implies SGI with the same constant. Due to the boundedness of the flipping rates \eqref{eq:bound-rates}, the standard Dirichlet form $\sum_{x\in\T_{L}}\langle(\nabla_{x}P^{\T_{L}}_{t}\sigma_{0})^{2}\rangle_{\T_{L}}$ is equivalent to the Dirichlet form \eqref{eq:Dirichlet} of our dynamics, hence we get \eqref{eq:var2}. Combine \eqref{eq:var1}, \eqref{eq:var2}, and $\lfloor\ell\rfloor^{\eta}\leq\ell^{2\eta}$, the integrand in the second term of \eqref{eq:up1} is upper bounded by $2t^{\alpha+1}D_{\T_{L}}(P^{\T_{L}}_{t}\sigma_{0})$. Thus in order to bound 
the two integrals in \eqref{eq:int}, it suffices to consider the first term in \eqref{eq:up1}. 

We then study the second order moment of the conditioned dynamics $\langle(B_{\partial\Lambda}P^{\T_{L}}_{t}\sigma_{x})^{2}\rangle_{\T_{L}}$. First, apply $x^{2}\leq2y^{2}+2(x-y)^{2}$ again, we get 
\begin{equation}\label{eq:up2}
   \langle(B_{\partial\Lambda}P^{\T_{L}}_{t}\sigma_{x})^{2}\rangle_{\T_{L}}\leq 2\langle\big[B_{\partial\Lambda}P^{\T_{L}}_{t}(B_{\partial\Lambda}\sigma_{x})\big]^{2}\rangle_{\T_{L}}+2\langle\big[B_{\partial\Lambda}P^{\T_{L}}_{t}(\sigma_{x}-B_{\partial\Lambda}\sigma_{x})\big]^{2}\rangle_{\T_{L}}.
\end{equation}
Since both the conditioning and the semigroup operators are $L^{2}$ contractions, the first term in \eqref{eq:up2} is bounded by $2\langle(B_{\partial\Lambda}\sigma_{x})^{2}\rangle_{\T_{L}}$. For the second term, we fix a given boundary condition $\omega$ and time $t\in[1,T]$, define $f_{t}^{\omega}$ to be the Radon-Nikodym density function:
\begin{equation}\label{eq:density}
    f^{\omega}_{t}:=\frac{d\mu_{\Lambda}(\cdot\vert\omega)P^{\T_{L}}_{t}}{d\mu_{\T_{L}}},
\end{equation}
where $d\mu_{\Lambda}(\cdot\vert\omega)P^{\T_{L}}_{t}$ denotes the law of the Glauber dynamics at time $t$, started from initial distribution $d\mu_{\Lambda}(\cdot\vert\omega)$.
This definition leads to
\begin{equation}
    \langle\big[B_{\partial\Lambda}P^{\T_{L}}_{t}(\sigma_{x}-B_{\partial\Lambda}\sigma_{x})\big]^{2}\rangle_{\T_{L}}=\langle\big(\E_{\T_{L}}\big[f^{\omega}_{t}\cdot(\sigma_{x}-B_{\partial\Lambda}\sigma_{x})\big]\big)^{2}\rangle_{\T_{L}}.
\end{equation}
To sum up, the first term in \eqref{eq:up1} is bounded by
\begin{equation}\label{eq:up3}
    \begin{aligned}
    \frac{4}{\vert\T_{L}\vert}\sum_{x\in\T_{L}}\int_{1}^{T}&(\alpha+1)t^{\alpha}\langle(B_{\partial\Lambda}\sigma_{x})^{2}\rangle_{\T_{L}}dt\\
    &+\frac{4}{\vert\T_{L}\vert}\sum_{x\in\T_{L}}\sum_{i=0}^{n(T)-1}\int_{t_{i}}^{t_{i+1}}(\alpha+1)t^{\alpha}\langle\big(\E_{\T_{L}}\big[f^{\omega}_{t}\cdot(\sigma_{x}-B_{\partial\Lambda}\sigma_{x})\big]\big)^{2}\rangle_{\T_{L}}dt.
    \end{aligned}
\end{equation}

For any fixed $\omega$ and $x\in\Lambda^{1}$, we have $(B_{\partial\Lambda}\sigma_{x})(\omega)=\E^{\omega}_{\T_{L}\setminus\partial\Lambda}(\sigma_{x})=\E^{\omega}_{\Lambda^{1}}(\sigma_{x})$. This is because the spin function $\sigma_{x}$ is only supported in the first cube $\Lambda^{1}$. By the FKG inequality and the symmetry of the torus, we obtain
\begin{equation}
    -\E^{+}_{\Lambda^{1}}(\sigma_{x})=\E^{-}_{\Lambda^{1}}(\sigma_{x})\leq\E^{\omega}_{\Lambda^{1}}(\sigma_{x})\leq\E^{+}_{\Lambda^{1}}(\sigma_{x}).
\end{equation}
Moreover, the cubes in the family $\Lambda$ are essentially equivalent, thus we get
\begin{align}\label{eq:condition-bound}
     \frac{1}{\vert\T_{L}\vert}\sum_{x\in\T_{L}}\langle(B_{\partial\Lambda}\sigma_{x})^{2}\rangle_{\T_{L}}&\leq\frac{1}{L^{d}}\Big(\sum_{x\in\Lambda^{1}}q\big[\E^{+}_{\Lambda^{1}}(\sigma_{x})\big]^{2}+\vert\partial\Lambda\vert\Big)\notag\\
    &\lesssim\frac{1}{\ell^{d}}\sum_{x\in\Lambda^{1}}\big[\E^{+}_{\Lambda^{1}}(\sigma_{x})\big]^{2}+\frac{1}{\ell}.
\end{align}
According to Lemma \ref{th:averaged}, the first term in \eqref{eq:condition-bound} is of order $O(\ell^{-(2\delta\wedge1)})$. By the definition of $\ell$, there exists a constant $c_{2}>0$ such that $\ell^{-(2\delta\wedge1)}t^{\alpha}\leq c_{2}$, and the first integral in \eqref{eq:up3} is therefore bounded by $4c_{2}(\alpha+1)(T-1)$.

Let $\cov_{\T
_{L}\setminus\partial\Lambda}^{\omega}$ be the covariance operator for the measure $\mu_{\T_{L}}(\cdot\vert\omega)$. By Jensen's inequality, we have
\begin{align}\label{eq:cov}
        \langle\big(\E_{\T_{L}}\big[f^{\omega}_{t}\cdot(\sigma_{x}-B_{\partial\Lambda}\sigma_{x})\big]\big)^{2}\rangle_{\T_{L}}&\leq\langle\big(B_{\partial\Lambda}\big[f^{\omega}_{t}\cdot(\sigma_{x}-B_{\partial\Lambda}\sigma_{x})\big]\big)^{2}\rangle_{\T_{L}}\notag\\
        &=\langle \big[\cov_{\T_{L}\setminus\partial\Lambda}^{\omega}(\sigma_{x},f^{\omega}_{t})\big]^{2}\rangle_{\T_{L}}.
\end{align}
To estimate the average of the covariance terms, we need the following Bodineau-Helffer inequality for the Ising model, which will be proved in the Appendix~\ref{app:BHI}.
\begin{lemma}
\label{th:BHI}
    Consider the Ising model \eqref{eq:free-ising} on a finite subgraph $G\subset\Z^{d}$ with any boundary condition $\omega$. Let $F:\Omega_{G}\mapsto\mathbb{R}_{+}$ be an arbitrary square-integrable non-negative function, we have
    \begin{equation}\label{eq:DBH}
        \sum_{x\in G}\cov^{\omega}_{G}(F,\sigma_{x})^{2}\leq\frac{64c_{M}}{\gamma_{G}^{2}}D^{\omega}_{G}(\sqrt{F}).
    \end{equation}
    In particular, $c_{M}$ denotes the $L^{\infty}$ upper bound for the flipping rates (see \eqref{eq:bound-rates}). $\cov^{\omega}_{G}(\cdot,\cdot)$ and $D^{\omega}_{G}(\cdot)$ denote the covariance and Glauber Dirichlet form with respect to the Ising measure on $G$. $\gamma_{G}$ refers to the log-Sobolev constant of the Ising model on $G$.
\end{lemma}
According to Lemma \ref{th:BHI} and \eqref{eq:cov}, and the fact that Dirichlet form is monotonically increasing in the size of the graphs, the second term of \eqref{eq:up3} is upper bounded by
\begin{align}\label{eq:entropy}
    \frac{256c_{M}c_{1}^{2}}{L^{d}}&\sum_{i=0}^{n(T)-1}\int_{t_{i}}^{t_{i+1}}(\alpha+1)t^{\alpha}\lfloor\ell\rfloor^{2\eta}\langle D_{\T_{L}}(\sqrt{f^{\omega}_{t}})\rangle_{\T_{L}}dt\notag\\
    &\leq\frac{256(\alpha+1)c_{M}c_{1}^{2}}{L^{d}}\sum_{i=0}^{n(T)-1}t_{i}^{\alpha}\lfloor\ell(t_{i})\rfloor^{2\eta}\langle \E_{\T_{L}}\Phi(f_{t_{i}}^{\omega})\rangle_{\T_{L}}.
\end{align}
Here, $\E_{\T_{L}}\Phi(f_{t_{i}}^{\omega}):=\E_{\T_{L}}(f_{t_{i}}^{\omega}\log f_{t_{i}}^{\omega})$ denotes the entropy of the non-negative function $f_{t_{i}}^{\omega}$. In the above, the second inequality follows from the de Bruijin identity: $\partial_{t}\E_{\T_{L}}\Phi(f^{\omega}_{t})=\int f^{\omega}_{t}\cL\log f^{\omega}_{t}d\mu_{\T_{L}}\leq-D_{\T_{L}}(\sqrt{f^{\omega}_{t}})$. Since the entropy of the Glauber dynamics is decreasing in time, we may trivially bound it by $\E_{\T_{L}}\Phi(f^{\omega}_{0})$. Therefore, we apply definition \eqref{eq:density} again and calculate
\begin{align}\label{eq:ent}
\E_{\T_{L}}\Phi(f^{\omega}_{0})&=\int d\mu_{\T_{L}}(\cdot\vert\omega)\log\frac{d\mu_{\T_{L}}(\cdot\vert\omega)}{d\mu_{\T_{L}}}\notag\\
        &=\int d\mu_{\T_{L}}(\cdot\vert\omega)\Big(\sum_{\substack{x\sim y\\x\in\partial\Lambda,\;y\notin\partial\Lambda}}(\sigma_{y}\omega_{x}-\sigma_{y}\sigma_{x})+\log\frac{Z}{Z^{\omega}}\Big),
\end{align}
where $Z$ and $Z^{\omega}$ are the partition functions for $\mu_{\T_{L}}$ and $\mu_{\T_{L}}(\cdot\vert\omega)$ respectively. For extra splittings, we define $Z^{\omega}_{j}$ to be the partition function for the partially conditioned measure $\mu_{\T_{L}}(\cdot\vert\omega\vert_{\partial\Lambda^{1}\cup\cdots\cup\partial\Lambda^{j}})$, $1\leq j\leq q$. \eqref{eq:ent} can be further bounded by 
\begin{equation}
\begin{aligned}
    \sup_{\omega\in\Omega_{\T_{L}}}\sum_{j=0}^{q-1}\log\frac{Z^{\omega}_{j}}{Z^{\omega}_{j+1}}&\leq\sum_{j=0}^{q-1}\log\Big(\sum_{\sigma_{x},\;x\in\partial\Lambda^{j+1}}e^{2d\beta_{c}}\Big)\leq c_{3}\sum_{j=0}^{q-1}\ell^{d-1}\leq c_{4}\ell^{-1}L^{d}.
\end{aligned}
\end{equation}
Thus, we get an entropy estimate $\langle \E_{\T_{L}}\Phi(f_{t_{i}}^{\omega})\rangle_{\T_{L}}\leq c_{4}\ell(t_{i})^{-1}L^{d}$ for all $i=1,...,n(T)-1$, and henceforth an upper bound for \eqref{eq:entropy}:
\begin{equation}\label{eq:entropy-bound}
    256c_{M}c_{1}c_{4}\sum_{i=0}^{n(T)-1}t_{i}^{\alpha+1}\ell(t_{i})^{-1}\leq 256c_{M}c_{1}c_{5}(T-1).
\end{equation}
This is due to $t_{i+1}-t_{i}>t_{i}$, as well as the fact that we can find a constant $c_{5}$ such that $c_{4}t_{i}^{\alpha}\ell(t_{i})^{-1}\leq c_{5}$. Finally, we insert the entropy bound \eqref{eq:entropy-bound} and the averaged 1-spin expectation bound \eqref{eq:condition-bound} into \eqref{eq:up3}. The proof is completed by choosing $C_{0}=\max\{4c_{2}(\alpha+1),256c_{M}c_{1}c_{5}\}$.
\end{proof}

\appendix
\section{Proof of Bodineau-Helffer inequality for discrete spins}
\label{app:BHI}

In this appendix, we present the proof of Lemma \ref{th:BHI}, which is similar to the diffusive case discussed in \cite[Lemma 5]{OR07}. Since the proof does not involve the geometry of the graph and the boundary condition, we will omit the $G$ subscripts and the $\omega$ superscripts for brevity in the following contexts. 

\begin{proof}[Proof of Lemma \ref{th:BHI}]
Without loss of generality, we assume $\int F(\sigma) d\mu=1$, where $\mu$ denotes the Ising measure on $G$. Since $\mu$ is the unique invariant measure of the dynamics, we know that $P_{\infty}F=P_{0}F=\int F(\sigma)d\mu=1$. By straightforward calculation, we have
    \begin{align}
        \cov(F,\sigma_{x})&=\int\sigma_{x}\big[(P_{0}F)(\sigma)-(P_{\infty}F)(\sigma)\big]d\mu\notag\\
        &=-\int\int_{0}^{\infty}\sigma_{x}(\cL P_{t}F)(\sigma)d\mu dt\notag\\
        &=-\int\int_{0}^{\infty}\sigma_{x}\;c(x,\sigma)\big[(P_{t}F)(\sigma^{x})-(P_{t}F)(\sigma)\big]d\mu dt.
    \end{align}
For the last equality, we applied the self-adjointness of the generator $\cL$ and canceled the sum over other vertices. Now we use the identity $a-b=(\sqrt{a}-\sqrt{b})(\sqrt{a}+\sqrt{b})$ for positive constants $a,b$ and the detailed balance condition for the flipping rates to get
\begin{align}\label{eq:balance}
    \vert\cov(F,\sigma_{x})\vert&\leq\int_{0}^{\infty}\int c(x,\sigma)\vert\nabla_{x}P_{t}F(\sigma)\vert d\mu dt\notag\\
    &=\int_{0}^{\infty}\int c(x,\sigma)\vert\nabla_{x}\sqrt{P_{t}F}(\sigma)\vert\times\vert\sqrt{P_{t}F(\sigma^{x})}+\sqrt{P_{t}F(\sigma)}\vert d\mu dt\notag\\
    &\leq2c_{M}^{1/2}\int^{\infty}_{0}\int c(x,\sigma)^{1/2}\vert\nabla_{x}\sqrt{P_{t}F}(\sigma)\vert\times\sqrt{P_{t}F(\sigma)}d\mu dt.
\end{align}
For any $\varepsilon\in(0,\gamma)$, we take the square and sum over $x\in G$ to get
\begin{align}\label{eq:calculation}
\sum_{x\in G}\cov(F,\sigma_{x})^{2}&\leq4c_{M}\sum_{x\in G}\Big(\int^{\infty}_{0}\big[\int c(x,\sigma)\vert\nabla_{x}\sqrt{P_{t}F}(\sigma)\vert^{2}d\mu\big]^{1/2}\big[\int (P_{t}F)(\sigma)d\mu\big]^{1/2}dt\Big)^{2}\notag\\
&=4c_{M}\sum_{x\in G}\frac{1}{\varepsilon^{2}}\Big(\int^{\infty}_{0}\varepsilon e^{\varepsilon t}\big[\int e^{-2\varepsilon t}c(x,\sigma)\vert\nabla_{x}\sqrt{P_{t}F}(\sigma)\vert^{2}d\mu\big]^{1/2}dt\Big)^{2}\notag\\
&\leq8c_{M}\sum_{x\in G}\frac{1}{\varepsilon^{2}}\int_{0}^{\infty}D(\sqrt{P_{t}F})de^{\varepsilon t}\notag\\
&\leq-8c_{M}\sum_{x\in G}\int_{0}^{\infty}\frac{1}{\varepsilon}e^{\varepsilon t}\frac{d}{dt}\E_{\mu}\Phi(P_{t}F) dt.
\end{align}
In particular, for the derivation of \eqref{eq:calculation}, we reapplied the de Bruijn identity: $\partial_{t}\E_{\mu}\Phi(P_{t}F)=\int P_{t}F\cL\log P_{t}Fd\mu\leq-D(\sqrt{P_{t}F})$. Thus, the inequality $\E_{\mu}\Phi(P_{t}F)\leq e^{-\gamma t}\E_{\mu}\Phi(F)$ (which follows from \eqref{eq:LSI}) and an easy integration by parts in time together imply 
\begin{align}
    \sum_{x\in G}\cov(F,\sigma_{x})^{2}&\leq\frac{8c_{M}}{\varepsilon}\E_{\mu}\Phi(F)+8c_{M}\int_{0}^{\infty} e^{\varepsilon t}e^{-\gamma t}\E_{\mu}\Phi(F) dt\notag\\
    &=8c_{M}\Big(\frac{1}{\varepsilon}+\frac{1}{\gamma-\varepsilon}\Big)\E_{\mu}\Phi(F).
\end{align}
Optimizing in $\varepsilon\in(0,\gamma)$ completes the proof of the inequality:
\begin{equation}
    \sum_{x\in G}\cov(F,\sigma_{x})^{2}\leq\frac{32c_{M}}{\gamma}\E_{\mu}\Phi(F)\leq\frac{64c_{M}}{\gamma^{2}}D(\sqrt{F}).
\end{equation}
\end{proof}
\begin{remark}
    A similar upper bound for the sum of squared covariances should also hold for stochastic dynamics in the continuum scaling limit, provided that the jump rates are bounded uniformly in the scaling parameter.
\end{remark}

\begin{ack}
    I wish to thank Roland Bauerschmidt and Benoit Dagallier for suggesting this problem, many helpful discussions, and proofreading on numerous drafts of this paper. In addition, I would like to thank Benoit Dagallier for explaining the proof of the Bodineau-Helffer inequality, as well as Romain Panis for the discussion on the high-dimensional critical two-point function. Last but not least, I must thank the two anonymous referees for giving valuable comments.
\end{ack}


\begin{thebibliography}{99}
\bibliographystyle{amsplain}

\bibitem{A82}
Aizenman, M.: Geometric analysis of {$\varphi ^{4}$} fields and {I}sing models.
  {I}, {II}. \emph{Commun. Math. Phys.}, \textbf{86}(1):1-48, 1982.

\bibitem{ADCS15}
Aizenman, M., Duminil-Copin, H., and Sidoravicius, V.: Random currents and continuity of Ising model's spontaneous magnetization. \emph{Commun. Math. Phys.}, \textbf{334}(2):719-742, 2015. 

\bibitem{AG83}
Aizenman, M. and Graham, R.:
On the renormalized coupling constant and the susceptibility in $\varphi ^{4}_{4}$ field theory and the {I}sing model in four dimensions. \emph{Nuclear Phys. B}, \textbf{225}(2):261-288, 1983. 

\bibitem{BBD23}
Bauerschmidt, R., Bodineau, T., and Dagallier, B.: Stochastic dynamics and the Polchinski equation: an introduction. \emph{Probab. Surv.}, \textbf{21}, 2024. 

\bibitem{BBS19}
Bauerschmidt, R., Brydges, D.C., and Slade, G.:
Introduction to a renormalisation group method, volume 2242 of \emph{Lecture Notes in Math.}. Springer, Singapore, 2019. 

\bibitem{BD22}
Bauerschmidt, R. and Dagallier, B.:
Log-{S}obolev inequality for near critical Ising models. \emph{Comm. Pure Appl. Math.}, \textbf{77}(4):2568-2576, 2024. 

\bibitem{CJN21}
Camia, F., Jiang, J., and Newman, C.M.: The effect of free boundary conditions on the
Ising model in high dimensions. \emph{Probab. Theory Related Fields}, \textbf{181}(1-3):311-328, 2021. 

\bibitem{CCR05}
Cancrini, N., Cesi, F., and Roberto, C.: Diffusive long-time behavior of Kawasaki dynamics. \emph{Electron. J. Probab.}, \textbf{10}:216-249, 2005. 

\bibitem{CM00}
Cancrini, N. and Martinelli, F.: 
On the spectral gap of Kawasaki dynamics under a mixing condition revisited. \emph{J. Math. Phys.}, \textbf{41}:1391-1423, 2000. 

\bibitem{DSS22}
Ding, J., Song, J., and Sun, R.: A new correlation inequality for {I}sing models with external fields. \emph{Probab. Theory Related Fields}, \textbf{186}(1):477-492, 2022. 

\bibitem{DCP24}
Duminil-Copin, H. and Panis, R.: New lower bounds for the (near) critical Ising and $\varphi^{4}$.
models’ two-point functions. \emph{Comm. Math. Phys.}, \textbf{406}(3):No.56, 2024.

\bibitem{FvdH}
Fitzner, R. and van der Hofstad, R.: Mean-field behavior for nearest-neighbor percolation
in $d>10$.
\emph{Electron. J. Probab.}, \textbf{22}(43):1–65, 2017. 

\bibitem{FILS78}
Fröhlich, J., Israel, R., Lieb, E.H., and Simon, B.: Phase transitions and reflection positivity. I. General theory and long range lattice
models. \emph{Comm. Math. Phys.}, \textbf{62}(1):1–34, 1978. 

\bibitem{FFS76}
Fröhlich, J., Simon, B., and Spencer, T.: Infrared bounds, phase
transitions and continuous symmetry breaking. \emph{Comm. Math. Phys.}, \textbf{50}(1):79–95, 1976. 

\bibitem{FV18}
Friedli, S. and Velenik, Y.: \emph{Statistical mechanics of lattice systems.}  A concrete mathematical introduction. Cambridge University Press, Cambridge, 2018. 

\bibitem{GZ03}
Guionnet, A. and Zegarlinski, B.: Lectures on logarithmic {S}obolev inequalities. In \emph{S\'eminaire de {P}robabilit\'es, {XXXVI}}, volume 1801 of \emph{Lecture Notes in Math.}, 1-134. Springer, Berlin, 2003. 

\bibitem{HHS19}
Handa, S., Heydenreich, M., and Sakai, A.: Mean-field bound on the 1-arm exponent for Ising ferromagnets in high dimensions. In \emph{Sojourns in Probability Theory and Statistical Physics-I: Spin Glasses and Statistical Mechanics, A Festschrift for Charles M. Newman.}, volume 298 of \emph{Springer Proc. Math. Stat.}, 183-198, Springer, Singapore, 2019. 

\bibitem{HH77}
Hohenberg, P.C. and Halperin, B.I.: Theory of dynamic critical phenomena. \emph{Rev. Mod. Phys.}, \textbf{49}(3): 435-479, 1977.

\bibitem{H15}
Hulshof, T.: The one-arm exponent for mean-field long-range percolation. \emph{Electron. J. Probab.}, \textbf{20}:1-26, 2015. 

\bibitem{JLQY99}
Janvresse, E., Landim, C., Quastel, J., and Yau, H.-T.: Relaxation to equilibrium of conservative dynamics I: Zero-range processes. \emph{Ann. Probab.}, \textbf{27}(1): 325-360, 1999. 

\bibitem{KN11}
Kozma, G. and Nachmias, A.: Arm exponents in high dimensional percolation. \emph{J. Amer. Math. Soc.}, \textbf{24}(2):375-409, 2011. 

\bibitem{LY03}
Landim, C. and Yau, H.-T.: Convergence to equilibrium of conservative particle systems on $\mathbb{Z}^{d}$. \emph{Ann. Probab.}, \textbf{31}(1):115-147, 2003. 

\bibitem{LF93}
Lauritsen K.B., and Fogedby, H.C.: Critical exponents from power spectra. \emph{J. Statist. Phys.}, \textbf{72}(1):189-205, 1993. 

\bibitem{LS12}
Lubetzky, E. and Sly, A.: Critical {I}sing on the square lattice mixes in polynomial time. \emph{Commun. Math. Phys.}, \textbf{313}(3):815-836, 2012. 

\bibitem{LS14}
Lubetzky, E. and Sly, A.: Cutoff for general spin systems with arbitrary boundary conditions. \emph{Comm. Pure Appl. Math.}, \textbf{67}:982-1027, 2014. 

\bibitem{LS13}
Lubetzky, E. and Sly, A.: Cutoff for the {I}sing model on the lattice. \emph{Invent. Math.}, \textbf{191}(3):719-755, 2013.

\bibitem{LS16}
Lubetzky, E. and Sly, A.: Information percolation and cutoff for the stochastic Ising model. \emph{J. Amer. Math. Soc.}, \textbf{29}:729-774, 2016. 

\bibitem{M99}
Martinelli, F.: Lectures on {G}lauber dynamics for discrete spin models. In \emph{Lectures on probability theory and statistics ({S}aint-{F}lour, 1997)}, volume 1717 of \emph{Lecture Notes in Math.}, 93-191. Springer, Berlin, 1999. 

\bibitem{MMS77}
Messager, A. and Miracle-Solé, S.: Correlation functions and boundary conditions in the Ising ferromagnet. \emph{J. Statist. Phys.}, \textbf{17}(4): 245-262, 1977. 

\bibitem{OR07}
Otto, F. and Reznikoff, M.G.: A new criterion for the logarithmic Sobolev inequality and two applications. \emph{J. Funct. Anal.}, \textbf{243}(1):121-157, 2007. 

\bibitem{Sak07}
Sakai, A.: Lace expansion for the Ising model. \emph{Commun. Math. Phys.}, \textbf{272}(2):283-344, 2007. 

\bibitem{LSC97}
Saloff-Coste, L.: Lectures on finite {M}arkov chains. In \emph{Lectures on probability theory and statistics ({S}aint-{F}lour, 1996)}, volume 1665 of \emph{Lecture Notes in Math.}, 301--413. Springer, 1997. 

\bibitem{WH}
Wang, F.-G. and Hu, C.-K.: Universality in dynamic critical phenomena. \emph{Phys. Rev. E}, \textbf{56}(2): 2310-2313, 1997.

\bibitem{WMTB}
Wu, T.T., McCoy, B.M., Tracy, C.A., and Barouch, E.: Spin-spin correlation functions for the
two-dimensional Ising model: Exact theory in the scaling region. \emph{Phys. Rev. B}, \textbf{13}:316–374, 1976.

\end{thebibliography}
\end{document}